\documentclass[11pt]{amsart}
\usepackage{amsthm,amssymb,amsmath}
\usepackage{graphicx}
\usepackage[usenames,dvipsnames]{color}

\newtheorem{theorem}{Theorem}[section]
\newtheorem{lemma}[theorem]{Lemma}

\newtheorem{proposition}[theorem]{Proposition}

\theoremstyle{definition}

\newtheorem{definition}[theorem]{\sc Definition}
\newtheorem{example}[theorem]{\bf Example}
\newtheorem{remark}[theorem]{\bf Remark}

\newcommand{\Id}{{\rm Id}}

\newcommand{\tr}{{\rm Tr}}

\begin{document}

\title[Kohn decomposition]{Kohn decomposition for forms on coverings of complex manifolds constrained along fibres}
%{Kohn decomposition for forms satisfying Banach space-type restrictions along fibres of a covering}



\author{A.~Brudnyi}

\address{Department of Mathematics and Statistics, University of Calgary, 
Calgary, Canada}

\email{abrudnyi@ucalgary.ca}

\author{D.~Kinzebulatov}

\address{The Fields Institute, Toronto, Canada}

\email{dkinzebu@fields.utoronto.ca}

\thanks{Research of the authors is partially supported  by NSERC}

\begin{abstract}
The classical result of J.J.~Kohn asserts that over a relatively compact subdomain $D$
with $C^\infty$ boundary of a Hermitian manifold whose Levi
form has at least $n-q$ positive eigenvalues or at least $q+1$ negative
eigenvalues at each boundary point, there are natural isomorphisms between 
the $(p,q)$ Dolbeault cohomology groups defined by means of $C^\infty$ up to the boundary differential forms on $D$
and the (finite-dimensional) spaces of harmonic $(p,q)$-forms on $D$ determined by the corresponding complex Laplace operator.
In the present paper, using Kohn's technique, we give a similar description of the $(p,q)$ Dolbeault cohomology groups of spaces of
differential forms taking values in certain (possibly infinite-dimensional) holomorphic Banach vector bundles on $D$. We apply this result
to compute the $(p,q)$ Dolbeault cohomology groups of some regular coverings of $D$ defined by means of $C^\infty$ forms constrained along fibres of the coverings.
\end{abstract}

\maketitle

\section{Introduction}
\label{intro}

Let $X$ be a connected Hermitian manifold of complex dimension $n$. A relatively compact subdomain $D=\{x\in X\, :\, \rho(x)<0\}\Subset X$, $\rho\in C^\infty(X)$, with $C^\infty$ boundary $\partial D$
is said to have {\em $Z(q)$-property}, if the Levi form of $\rho$
has at least $n-q$ positive eigenvalues or at least $q+1$ negative eigenvalues at each boundary point of $D$ (e.g., a strongly pseudoconvex subdomain of $X$ has $Z(q)$-property for all $q>0$). 

Let $\Lambda^{p,q}(\bar{D})$ be the space of $C^\infty$ ($p,q$)-forms on $D$ that admit $C^\infty$ extension in some open neighbourhood of the closure $\bar{D}$ of $D$ in $X$. Using the Hermitian metric on $X$, in a standard way one defines the Laplace operator $\Box$ on
$\Lambda^{p,q}(\bar{D})$, see, e.g., \cite{K} for details. The forms in ${\rm Ker}\,\Box=:\mathcal H^{p,q}(\bar{D})$ are called {\em harmonic}.

%Spaces $\Lambda^{p,q}(\bar{D})$ form a complex under the action of operator $\bar{\partial}:\Lambda^{p,q}(\bar{D}) \rightarrow \Lambda^{p,q+1}(\bar{D})$.

The following result is the major consequence of the theory developed by J.J.~Kohn, see \cite{KN}, \cite{K} or \cite{FK}.

\begin{theorem}%[\cite{FK}]
\label{thm0}
Suppose $D$ has $Z(q)$-property. Then ${\rm dim}_{\mathbb C}\mathcal H^{p,q}(\bar{D})<\infty$ and each $\bar{\partial}$-closed form
$\omega\in \Lambda^{p,q}(\bar{D})$ is uniquely presented as
\begin{equation}
\label{kohnd}
\omega=\bar{\partial}\xi+\chi,\quad\text{where}\quad \xi \in \Lambda^{p,q-1}(\bar{D}),\ \chi\in\mathcal H^{p,q}(\bar{D}).
\end{equation}
\end{theorem}

It follows that the map
$$
\mathcal H^{p,q}(\bar{D}) \ni  \omega \mapsto [\omega] \in H^{p,q}(\bar{D}):=\{\omega \in \Lambda^{p,q}(\bar{D}): \bar{\partial} \omega=0\}/\bar{\partial} \Lambda^{p,q-1}(\bar{D}),
$$
where $[\omega]$ stands for the cohomology class of $\omega$,
is an isomorphism.

As a corollary, one obtains the characterization of the \textit{Dirichlet cohomology groups}
$$
H_0^{r,s}(\bar{D}):=\mathcal Z_0^{r,s}(\bar{D})/\mathcal B_0^{r,s}(\bar{D}),
$$
where
$$
\mathcal Z_0^{r,s}(\bar{D}):=\{\omega \in \Lambda_0^{r,s}(\bar{D}): \bar{\partial}\omega=0 \}, \quad \mathcal B_0^{r,s}(\bar{D}):=\bar{\partial}\{\omega \in \Lambda^{r,s-1}_0(\bar{D}): \bar{\partial}\omega \in \Lambda_0^{r,s}(\bar{D})\}
$$
and
$$
\Lambda_0^{r,s}(\bar{D}):=\{\omega \in \Lambda^{r,s}(\bar{D}): \omega|_{\partial D}=0\}.
$$
%Namely, one has the following result (cf.~\cite{FK} -- from now on we will be referring to this monograph, which contains the results of \cite{KN}):
Namely, one has the following result:

\begin{theorem}[\cite{FK}]
\label{thm0_2}
If $D$ has $Z(q)$-property, then there is a natural isomorphism 
$$H_0^{n-p,n-q}(\bar{D}) \cong (H^{p,q}(\bar{D}))^{\ast}$$
%(here $V^\ast$ denotes the dual space of $V$) 
induced by the map associating to each $\xi \in \mathcal Z_0^{n-p,n-q}(\bar{D})$ the linear functional
$$
\mathcal Z^{p,q}(\bar{D}) \ni \theta \mapsto \int_D \theta \wedge \xi.
$$
\end{theorem}

Since Theorems \ref{thm0} and \ref{thm0_2} are independent of $p$, they can be viewed as assertions about spaces of $C^\infty$ $(0,q)$-forms on $\bar D$ with values in the (finite-dimensional) holomorphic vector bundle of $(p,0)$-forms on $X$. This manifests a more general fact: Kohn's arguments can be transferred without significant changes to spaces of $C^\infty$  $(p,q)$-forms on $\bar D$ taking values in a finite-dimensional Hermitian holomorphic vector bundle on $X$ (see, e.g.,~\cite[Ch.IV]{FK}).

The goal of the present paper is to extend Theorems \ref{thm0} and \ref{thm0_2} to spaces of $C^\infty$ $(p,q)$-forms on $\bar D$ with values in an \textit{infinite-dimensional} holomorphic Banach vector bundle $E$ on $X$. (Note that if $E$ is not Hilbertian, Kohn's arguments are not applicable.) %Our motivating example (Section \ref{motivation}) 
We apply these results to differential forms on (possibly unbounded!) subdomains $\bar{D}'=r^{-1}(\bar{D}) \subset X'$, where $r:X'\rightarrow X$ is a regular covering of a complex manifold $X$, 
%and $p:X' \rightarrow X$ is a holomorphic regular covering; these forms 
satisfying additional constraints along fibres of the covering (see Section~3).
Such forms appear within theories of algebras of bounded holomorphic functions on regular coverings of $X$. Another, sheaf-theoretic, approach to the study of such algebras was proposed in \cite{BrK}. It is based on analogues of Cartan theorems A and B for coherent-type sheaves on certain fibrewise complactifications of the covering (a topological space having some properties of a complex manifold).% which allowed us to extend the basic results of complex function theory to these forms. Below we elaborate an alternative approach to study of these forms.

\section{Main results}
\label{main}

%We extend Theorems \ref{thm0} and \ref{thm0_2} to the spaces $\Lambda^{p,q}(\bar{D},E)$ of $C^\infty$ forms on $\bar{D}$ taking values in a holomorphic Banach vector bundle 
Let $\pi:E \rightarrow X$ be a holomorphic Banach vector bundle with fibre $B$. %Let $GL(B)$ denote the group of invertible bounded linear operators on $B$. 
For an open $U\subset X$ by $\Lambda^{p,q}(U,E)$ we denote the space of $C^\infty$ $E$-valued $(p,q)$-forms on $U$, i.e., $C^\infty$ sections of the holomorphic Banach vector bundle $E \otimes\bigl( \wedge^p T^*X\bigr) \wedge\bigl(\wedge^q \overline{T^*X}\bigr)$ over $U$ (here $T^*X$ is the holomorphic cotangent bundle on $X$).
Also, we denote by $\mathcal O(X, E)$ the space of holomorphic sections of $E$ equipped with (Hausdorff) topology of uniform convergence on compact subsets of $X$ (defined in local trivializations on $E$ by the norm of $B$). For  a compact subset $S\subset X$ by $C(S,E)$ we denote the space of continuous sections of $E$ on $S$ equipped with topology of uniform convergence.
%$\mathcal O(\bar D, E)$ is the space of holomorphic sections of $E$ over $D$ continuous on $\bar D$ equipped with topology of uniform convergence on $\bar D$. 
(The former space admits the natural structure of a Fr\'{e}chet space and the latter one of a complex Banach space).

Let $\Lambda^{p,q}(\bar{D},E):=\Lambda^{p,q}(X,E)|_{\bar{D}}$ be the space of restrictions to $\bar{D}$ of $C^\infty$ $E$-valued forms on $X$. In a standard way, using local trivializations on $E$, we equip $\Lambda^{p,q}(\bar{D},E)$ with the Fr\'{e}chet topology determined by a sequence of $C^k$-like norms $\{\|\cdot\|_{k,\bar{D},E}\}_{k=0}^\infty$ (see subsection \ref{results}).
Then the standard operator 
$$\bar{\partial}:\Lambda^{p,q}(\bar D,E) \rightarrow \Lambda^{p,q+1}(\bar D,E)$$ 
is continuous.
Consider the corresponding subspaces of $\bar{\partial}$-closed and $\bar{\partial}$-exact forms
$$\mathcal Z^{p,q}(\bar{D},E):=\{\omega \in \Lambda^{p,q}(\bar{D},E): \bar{\partial} \omega=0\}\quad\text{and}\quad
\mathcal B^{p,q}(\bar{D},E):=\bar{\partial}\Lambda^{p,q-1}(\bar{D},E)$$
equipped with topology induced from $\Lambda^{p,q}(\bar{D},E)$.

Our results concern the structure of the cohomology group
$$
H^{p,q}(\bar{D},E):=\mathcal Z^{p,q}(\bar{D},E)/\mathcal B^{p,q}(\bar{D},E)
$$
and its dual, for bundles from the class $\Sigma_0(X)$ consisting of {\em direct summands of holomorphically trivial bundles}, that is, $E\in \Sigma_0(X)$ if
there exists a holomorphic Banach vector bundle $E'$ on $X$ such that the Whitney sum of bundles $E\oplus E'$ is holomorphically trivial.

\begin{example} 
\label{sigma_ex}
Each holomorphic Banach vector bundle on a Stein manifold $Y$ is in $\Sigma_0(Y)$ (see, e.g.,~\cite[Th.~3.9]{Obz}). Thus if $f: X\rightarrow Y$ is a holomorphic map, then $E:=f^*E'\in \Sigma_0(X)$ for every holomorphic Banach vector bundle $E'$ on $Y$. The class of such bundles $E$ will be denoted by $\Sigma_0^s(X)$.
\end{example}

In what follows, by $Z^m$ we denote the $m$-fold direct sum of a vector space $Z$, and we ignore all objects related to $m=0$.
\begin{theorem}
\label{thm2}
Suppose $E \in \Sigma_0(X)$ and $D\Subset X$ has $Z(q)$-property. Fix a basis $\{\chi_i\}_{i=1}^m\subset \mathcal H^{p,q}(\bar{D})$.

(1) There exist a closed complemented subspace
$\mathcal A \subset \mathcal O(X, E)^m$ and a finite subset $S\subset \bar D$ such that 
\begin{itemize}
\item[(a)]
$\mathcal A|_{S}$ is a closed %complemented 
subspace of the Banach space $C(S,E)^m$ and
the restriction to $S$ induces an isomorphism of the Fr\'{e}chet spaces
$\mathcal A\cong\mathcal A|_{S}$;
\item[(b)]
The linear map $L:\mathcal B^{p,q}(\bar D, E)\oplus \mathcal A\rightarrow \mathcal Z^{p,q}(\bar{D}, E)$,
\[
L\bigl(\eta, (f_1,\dots,f_m)\bigr):=\eta+\sum_{i=1}^m f_i|_{\bar{D}}\cdot\chi_i,\quad \eta\in \mathcal B^{p,q}(\bar D, E),\ (f_1,\dots,f_m)\in \mathcal A,
\]
is an isomorphism of Fr\'{e}chet spaces.
\end{itemize} 

(2) If the group $GL(B)$ of invertible bounded linear operators on the fibre $B$ of $E$
is contractible, and $E\in\Sigma_0^s(X)$, then the restriction map $r_x:\mathcal A\rightarrow \pi^{-1}(x)^m\cong B^m$, $(f_1,\dots, f_m)\mapsto (f_1(x),\dots, f_m(x))$, is a Banach space isomorphism for each $x\in X$.\medskip

\end{theorem}

\begin{remark}\label{rem1}
(1)~It follows that $\mathcal B^{p,q}(\bar D, E)$ is a closed subspace of the  Fr\'{e}chet space $\mathcal Z^{p,q}(\bar{D}, E)$ and so the quotient space $H^{p,q}(\bar{D},E)$ is Fr\'{e}chet. It is trivial if $m=0$, for otherwise, it is isomorphic (in the category of  Fr\'{e}chet spaces) to the complex Banach space $\mathcal A|_{S}\subset C(S,E)^m \cong B^{rm}$; here $r$ is the cardinality of $S$.

\smallskip

(2)~If $X$ is a Stein manifold, then it admits a K\"{a}hler metric. Working with this metric, one obtains that the corresponding harmonic forms $\chi_i$ in Theorem \ref{thm2} are also $d$-closed (see, e.g.~\cite[Ch.0, Sect.7]{GH}).

\smallskip

(3)~The class of complex Banach spaces $B$ with contractible group $GL(B)$ include infinite-dimensional Hilbert spaces, spaces $\ell^p$ and $L^p[0,1]$, $1\le p\le \infty$, $c_0$ and $C[0,1]$, spaces $L_p(\Omega,\mu)$, $1<p<\infty$, of $p$-integrable measurable functions on an arbitrary
measure space $\Omega$, some classes of reflexive symmetric function spaces and spaces $C(G)$ for $G$ being infinite dimensional compact topological groups (see, e.g., \cite{M} for details).

\end{remark}

Next, we formulate an analogue of Theorem \ref{thm0_2}. 
We will need the following notation.
Let  $V\rightarrow X$  be a holomorphic Banach vector bundle.
Set 
$$
\Lambda_0^{r,t}(\bar{D},V):=\{\omega \in \Lambda^{r,t}(\bar{D},V): \omega|_{\partial D}=0\}
$$
and define the \textit{$V$-valued Dirichlet cohomology groups} of $\bar{D}$ by the formula
$$
H_0^{r,s}(\bar{D},V):=\mathcal Z_0^{r,s}(\bar{D},V)/\mathcal B_0^{r,s}(\bar{D},V),
$$
where 
\[
\begin{array}{l}
Z_0^{r,s}(\bar{D},V):=\{\omega \in \Lambda_0^{r,s}(\bar{D},V): \bar{\partial}\omega=0 \}\qquad\text{and} \medskip\\
B_0^{r,s}(\bar{D},V):=\bar{\partial}\{\omega \in \Lambda^{r,s-1}_0(\bar{D},V):  \bar{\partial}\omega \in \Lambda_0^{r,s}(\bar{D},V)\}.
\end{array}
\]
We endow spaces $B_0^{r,s}(\bar{D},V) \subset Z_0^{r,s}(\bar{D},V) \subset \Lambda_0^{r,s}(\bar{D},V)$ with the topology induced by that of $\Lambda^{r,s}(\bar{D},V)$. One can easily check that
$Z_0^{r,s}(\bar{D},V)$ and $\Lambda_0^{r,s}(\bar{D},V)$ are Fr\'{e}chet spaces with respect to this topology.

We retain notation of Theorem \ref{thm2}. In the following result $H^{p,q}(\bar{D},E)$, $E\in\Sigma_0(X)$, is equipped with the Fr\'{e}chet space structure given by Theorem \ref{thm2}. By $E^*$ we denote the bundle dual to $E$. Also, for $m>0$,
$\{\chi_i\}_{i=1}^m$ is a fixed basis of $\mathcal H^{p,q}(\bar{D})$ and $\mathcal A \subset \mathcal O(X, E)^m$ is the corresponding subspace of Theorem \ref{thm2}.

%{\bf Here if we are working with the basis of Theorem 2.1, then introduce  the dual basis explicitly using the wedge product of Theorem 1.2.}

\begin{theorem}
\label{thm6}
Suppose $E \in \Sigma_0(X)$ and $D\Subset X$ has $Z(q)$-property. Fix forms
$\{\gamma_i\}_{i=1}^m \subset \mathcal Z_0^{n-p,n-q}(\bar{D})$ such that $\int_D \chi_i \wedge \gamma_j=\delta_{ij}$ - the Kronecker delta. (Their existence follows from Theorem \ref{thm0_2}.)

\smallskip

(1) $\mathcal B_0^{n-p,n-q}(\bar{D},E^*)$ is a closed subspace of the  Fr\'{e}chet space $\mathcal Z_0^{n-p,n-q}(\bar{D},E^*)$; moreover, the  quotient (Fr\'{e}chet) space $H_0^{n-p,n-q}(\bar{D},E^*)$ is  naturally isomorphic to the dual space $\bigl(H^{p,q}(\bar{D},E)\bigr)^*$. 

\smallskip

(2) There exist a closed subspace $\mathcal B \subset \mathcal O(X, E^*)^m$ isomorphic to the dual of $\mathcal A$ and a finite subset $S^*\subset\bar{D}$ such that 
\begin{itemize}
\item[(a)] The restriction to $S^*$ induces an isomorphism of the Fr\'{e}chet spaces
$\mathcal B\cong\mathcal B|_{S^*}$;
\item[(b)]
The linear map $M:\mathcal B \rightarrow H_0^{n-p,n-q}(\bar{D},E^*)$
$$
M(h_1,\dots,h_m):=\left[\sum_{i=1}^m h_i|_{\bar D} \cdot \gamma_i\right], \quad (h_1,\dots,h_m) \in \mathcal B,
$$
is an isomorphism of Fr\'{e}chet spaces; here $[\eta]$ stands for the cohomology class of $\eta$.
\end{itemize}
\end{theorem}

The isomorphism in (1) is induced by the map associating to each $\xi \in \mathcal Z_0^{n-p,n-q}(\bar{D},E^*)$ a linear functional
$$
\mathcal Z^{p,q}(\bar{D},E) \ni \theta \mapsto J_E(\theta,\xi),
$$
where 
$$
J_E:\Lambda^{p,q}(\bar{D},E) \times \Lambda^{n-p,n-q}(\bar{D},E^*) \rightarrow \mathbb C
$$ 
is a certain continuous bilinear form, see Section 5 below. (In particular, if $E:=X\times\mathbb C$, $J_E(\theta,\xi):=\int_D \theta\wedge\xi$.)

\begin{remark}
It is not clear yet to what extent assertions of Theorems \ref{thm2} and \ref{thm6} are valid for holomorphic Banach vector bundles on $X$ not in $\Sigma_0(X)$. In particular, is it true that in this general setting spaces $H^{p,q}(\bar D,E)$ are Hausdorff (in the corresponding quotient topologies), and what can be said about the Serre-type duality between
 $H^{p,q}(\bar D,E)$  and  $H_0^{n-p,n-q}(\bar D,E^*)\, $?
\end{remark}

\section{Applications}

%\subsection{Motivating example}
\label{motivation}
As it was mentioned in the Introduction, forms taking values in holomorphic Banach vector bundles arise as an equivalent presentation of forms defined on subdomains of coverings of complex manifolds and satisfying additional constraints along the fibres of the coverings. In what follows, we outline the main features of this construction (see \cite{Br}, \cite{BrK} for details).

Let $r:X'\rightarrow X$ be a regular covering with a deck transformation group $G$ of a connected complex manifold $X$. Assume that $X'$ is equipped with  a path metric $d'$ determined by the pullback to $X'$ of a smooth hermitian metric on $X$.

\begin{definition}\label{def1}
By $C_B(X')=C_B(X',X,r)$ we denote the space of complex continuous functions $f:X' \rightarrow \mathbb C$ 
uniformly continuous with respect to metric $d'$ on subsets $r^{-1}(U)$, $U \Subset X$, and such that
for each $x \in X'$ functions $G \ni g \mapsto f(g \cdot x)$
belong to a complex Banach space $B$
of functions $u:G \rightarrow \mathbb C$ such that
$$
u \in B,~~g \in G \quad \Rightarrow \quad R_g u \in B,\quad\text{where}\quad
R_g(u)(h):=u(h g)\ (h\in G),
$$
and each $R_g$ is an invertible bounded linear operator on $B$. 
\end{definition}

Here are some examples of such spaces $B$.
\begin{example}
\label{ex2}
\textit{Uniform algebras.~} As space $B$ one can take a closed unital 
subalgebra of the algebra $\ell_\infty(G)$ of bounded complex functions on $G$ (with pointwise multiplication and $\sup$-norm) invariant with respect to the action of $G$ on $\ell_\infty(G)$ by right translations $R_g$, $g \in G$, e.g.,~algebra $\ell_\infty(G)$ itself, algebra $c(G)$ of bounded complex functions on $G$ that admit continuous extensions to the one-point compactification of group $G$, algebra  $AP(G)$ of the von Neumann almost periodic functions on group $G$ (i.e.~uniform limits on $G$ of linear combinations of matrix elements of irreducible unitary representations of $G$), etc. 
If group $G$ is finitely generated, then in addition to $c(G)$ one can take subalgebras $c_{E}(G) \subset \ell_\infty(G)$ of functions having limits at `$\infty$' along each `path' (see \cite{BrK} for details).

\smallskip

\noindent\textit{Orlicz spaces.~}Let $\mu$ be a $\sigma$-finite regular Borel measure on $G$ such that for each $g\in G$ there exists a constant $c_g>0$ so that $\mu(h\cdot g)\le c_g\cdot\mu (h)$ for all $h\in G$. Let $\Phi: [0,\infty)\to [0,\infty)$ be a convex function such that
\[
\lim_{x\to\infty}\frac{\Phi(x)}{x}=\infty\qquad \text{and}\qquad \lim_{x\to 0^+}\frac{\Phi(x)}{x}=0.
\]
%(``Young function'').
As space $B$ one can take the space $\ell_\Phi$ of complex $\mu$-measurable functions on $G$ such that $\int_G\Phi(|f|)d\mu<\infty$ endowed with norm
\[
\|f\|_{\Phi}:=\inf\left\{C\in (0,\infty)\, :\, \int_G\Phi\left(\frac{|f|}{C}\right)d\mu\le 1\right\}.
\]
If $\Phi(t):=t^p$, $1< p<\infty$, then one obtains classical spaces $\ell^p(G,\mu)$. 

As measure $\mu$ one can take, e.g., the counting measure $\mu_c$ on $G$, in which case all $c_g=1$. 
If group $G$ is finitely generated, one can take $\mu:=e^u\mu_c$, where $u:G\to \mathbb R$ is a uniformly continuous function with respect to the $G$-invariant metric on $G$ induced by the natural metric on the Cayley graph of $G$ defined by a fixed family of generators of $G$. 

\end{example}

It is easily seen that the definition of space $C_B(X')$ does not depend on the choice of the hermitian metric on $X$.
If we fix a cover $\mathcal U$ of $X$ by simply connected relatively compact coordinate charts and for a given chart $U \in \mathcal U$ endow the `cylinder' $U':=r^{-1}(U)$ with local coordinates pulled back from $U$ (so that in these coordinates $U'$ is naturally identified with $U \times G$), then every function $f$ in $C_B(X')$, restricted to $U'$, can be viewed as a continuous function on $U$ taking values in space $B$. 

We equip $C_B(X')$ with the Fr\'{e}chet topology defined by the family of seminorms $\|\cdot\|_{U}$, $U\Subset X$,
\[
\|f\|_{U}:=\sup_{x\in U}\|f_x\|_B, \quad f\in C_B(X'),
\]
where $f_x(g):=f(g\cdot x)$, $g\in G$, and $\|\cdot\|_B$ is the norm of $B$.

By $\mathcal O_B(X'):=C_B(X') \cap \mathcal O(X')$ we denote the subspace of holomorphic functions in $C_B(X')$.

\begin{example}[Bohr's almost periodic functions, see, e.g.,~\cite{BrK} for details]
\label{ex}
A tube domain $T'=\mathbb R^n+i\Omega \subset \mathbb C^n$, where $\Omega \subset \mathbb R^n$ is open and convex, can be viewed as a regular covering $r:T' \rightarrow T\, (:=r(T') \subset \mathbb C^n)$ with deck transformation group $\mathbb Z^n$, where
\begin{equation*}
r(z):=\bigl(e^{i z_1}, \dots, e^{i z_n}\bigr), \quad z=(z_1,\dots,z_n) \in T'.
\end{equation*}
Let $B=AP(\mathbb Z^n)$ be the complex Banach algebra of the von Neumann almost periodic functions on group $\mathbb Z^n$ endowed with $\sup$-norm. Then $\mathcal O_B(T')\, (=:\mathcal O_{AP}(T'))$ coincides with the algebra of holomorphic almost periodic functions on $T'$, i.e.~uniform limits on tube subdomains $T''=\mathbb R^n+i\Omega''$ of $T'$, $\Omega'' \Subset \Omega$, of exponential polynomials
\begin{equation*}
z\mapsto\sum_{k=1}^m c_ke^{i \langle z,\lambda_k\rangle}, \quad z\in T',\quad c_k \in \mathbb C, \quad \lambda_k \in \mathbb R^n,
\end{equation*}
where $\langle\cdot,\cdot\rangle$ is the Hermitian inner product on $\mathbb C^n$,
and
$C_{B}(T')\, (=:C_{AP}(T'))$ coincides with the algebra of continuous uniformly almost periodic functions on $T'$. 
%-- the uniform limits (on tube subdomains $T''$ as above) of 
%\begin{equation*}
%z\mapsto\sum_{k=1}^m c_k\bigl(\Imag(z)\bigr)e^{i \langle z,\lambda_k\rangle}, \quad z\in T',\quad c_k \in C(\Omega), \quad \lambda_k \in \mathbb R^n, \quad \Imag(z)=(\Imag(z_1),\dots,\Imag(z_n)) \in \Omega.
%\end{equation*}
\end{example}

The theory of almost periodic functions was created in the 1920s by H.~Bohr and nowadays is widely used in
various areas of mathematics including number theory, harmonic analysis, differential equations (e.g.,~KdV equation), etc. 
We are interested, in particular, in studying cohomology groups of spaces of differential forms with almost periodic coefficients. Such forms arise as the special case of the following

\begin{definition}
By $\Lambda^{p,q}_B(X')=\Lambda^{p,q}_B(X',X,r)$ we denote the subspace of $C^\infty$ $(p,q)$-forms $\omega$ on $X'$ such that in each `cylinder' $U'=r^{-1}(U)\, (\cong U\times G)$, $U \in \mathcal U$, in local coordinates pulled back from $U$,
$$
\omega|_{U'}(z,\bar{z},g)=\sum_{|\alpha|=p, \,|\beta|=q} f_{\alpha,\beta}(z,\bar{z},g)\, dz_\alpha \wedge d\bar{z}_\beta,
$$
where $U \ni z \mapsto f_{\alpha,\beta}(z,\bar{z},\cdot)$ are Fr\'{e}chet $C^\infty$ $B$-valued functions (cf.~subsection \ref{results}).

\smallskip

For a subdomain $D\Subset X$ we set $D':=r^{-1}(D)$ and
$\Lambda_B^{p,q}(\bar{D}'):=\Lambda_B^{p,q}(X')|_{\bar{D}'}$.
\end{definition}

Comparing definitions of spaces $\Lambda_B^{p,q}(\bar{D}')$ and $\Lambda^{p,q}(\bar{D},E_{X'})$, where $\pi:E_{X'} \rightarrow X$ is the holomorphic Banach vector bundle with fibre $B$ associated to regular covering $r:X' \rightarrow X$ (viewed as a principal bundle on $X$ with fibre $G$, see e.g.,~\cite{BrK}), and likewise endowing $\Lambda_B^{p,q}(\bar{D}')$ with a sequence of $C^k$-like seminorms, we obtain isomorphisms of Fr\'{e}chet spaces
\begin{equation}
\label{isom}
\Lambda^{p,q}_B(\bar{D}') \cong \Lambda^{p,q}(\bar{D},E_{X'})
\end{equation}
commuting with the corresponding $\bar\partial$ operators. These induce (algebraic) isomorphisms of the corresponding cohomology groups:
\begin{equation}\label{iso}
H^{p,q}_{B}(\bar{D}')\cong H^{p,q}(\bar D,E_{X'}),
\end{equation}
where 
\[
\begin{array}{c}
\displaystyle
H^{p,q}_{B}(\bar{D}'):=\mathcal Z^{p,q}_{B}(\bar{D}')/\mathcal B^{p,q}_{B}(\bar{D}');\\
\\
\mathcal Z^{p,q}_{B}(\bar{D}'):=\{\omega \in \Lambda^{p,q}_{B}(\bar{D}')\,:\, \bar{\partial} \omega=0 \},\qquad
\mathcal B^{p,q}_{B}(\bar{D}'):=\bar{\partial} \Lambda^{p,q-1}_{B}(\bar{D}').
\end{array}
\]

\smallskip

Now, suppose that $D\Subset X$ has $Z(q)$-property. Let $f:X\rightarrow Y$ be a holomorphic map into a connected Stein manifold $Y$. Then $f$ induces a homomorphism of fundamental groups $f_*:\pi_1(X)\rightarrow\pi_1(Y)$. Without loss of generality, we may and will assume that $f_*$ is an epimorphism. (Indeed, if $H:=f_*(\pi_1(X))$ is a proper subgroup of $\pi_1(Y)$, then by the covering homotopy theorem,
there exist an unbranched covering $p: Y'\rightarrow Y$ such that $\pi_1(Y')=H$, and a holomorphic map $f':X\rightarrow Y'$ such that $f=p\circ f'$.
Moreover, $Y'$ is Stein. Thus, we may replace $f$ by $f'$.) \smallskip

Next, let $r:X'\rightarrow X$ be a regular covering with a deck transformation group $G$ isomorphic to a quotient group of $\pi_1(Y)$. If
$\tilde r: Y'\rightarrow Y$ is the regular covering of $Y$ with the deck transformation group $G$, then by the covering homotopy theorem
there exists a holomorphic map $f': X'\rightarrow Y'$ such that $f\circ r=\tilde r\circ f'$. This implies that $E_{X'}=f^*E_{Y'}$ (here $E_{Y'}\rightarrow Y'$ is the holomorphic Banach vector bundle with fibre $B$ defined similarly to $E_{X'}$ above). In particular, $E_{X'}\in \Sigma_0^s(X)$, see Example \ref{sigma_ex}, and hence Theorem \ref{thm2} can be applied to describe cohomology groups $H^{p,q}_{B}(\bar{D}')$. 
Under the above assumptions we obtain (as before, we ignore all objects related to $m=0$): 

\begin{theorem}\label{te3.5}
Let $\{\chi_i'\}_{i=1}^m$ be the pullback to $\bar D'$ of a basis in $\mathcal H^{p,q}(\bar{D})$.

(1) There exist a closed complemented subspace
$\mathcal A \subset \mathcal O_B(X')^m$ and a finite subset $S\subset \bar D$ such that 
\begin{itemize}
\item[(a)]
$\mathcal A|_{S'}$, $S':=r^{-1}(S)$, is a closed 
subspace of the Banach space $(C_B(X')|_{S'})^m\cong B^{cm}$, $c:={\rm card}\, S$, and
the restriction $\mathcal A\rightarrow \mathcal A|_{S'}$ is an isomorphism of Fr\'{e}chet spaces;
\item[(b)]
$\mathcal B_B^{p,q}(\bar D)$ is a closed subspace of the  Fr\'{e}chet space $\mathcal Z_B^{p,q}(\bar{D})$ and the linear map $L:\mathcal B_B^{p,q}(\bar D)\oplus \mathcal A\rightarrow \mathcal Z_B^{p,q}(\bar{D})$,
\[
L\bigl(\eta, (f_1,\dots,f_m)\bigr):=\eta+\sum_{i=1}^m f_i|_{\bar{D'}}\cdot\chi_i',\quad \eta\in \mathcal B_B^{p,q}(\bar D),\ (f_1,\dots,f_m)\in \mathcal A,
\]
is an isomorphism of Fr\'{e}chet spaces.
\end{itemize}

(2) If the group $GL(B)$ of invertible bounded linear operators on $B$ 
is contractible, then the restriction map $\mathcal A\rightarrow \mathcal A|_{\pi^{-1}(x)}\cong B^m$ is a Banach space isomorphism for each $x\in X$.
\end{theorem}
\begin{remark}\label{rem3.6}
(1)~The result shows that $H_B^{p,q}(\bar{D})$ is a Fr\'{e}chet space, trivial if $m=0$ and isomorphic to a closed subspace of the Banach space $B^{cm}$ otherwise.\smallskip

(2)~As follows from the assumptions, Theorem \ref{te3.5} is applicable to nontrivial coverings $r:X'\rightarrow X$ provided that $X$ admits a holomorphic map into a Stein manifold that induces a nontrivial homomorphism of the corresponding fundamental groups. In particular,
if $X$ is Stein, the theorem is valid for any regular covering $r:X'\rightarrow X$. If, in addition, $D$ is homotopically equivalent to $X$, then $H_B^{p,q}(\bar{D})=0$ for $p+q>n:={\rm dim}\, X$. Indeed, in this case, due to Remark \ref{rem1}\,(2), $\mathcal H^{p,q}(\bar{D})$ has a basis consisting of $d$-closed forms. Since $X$, being Stein, is homotopically equivalent to an $n$-dimensional CW-complex, these forms must be $d$-exact for $p+q>n$ and, hence, equal to zero (because they are harmonic with respect to the Laplacian defined by $d$). This implies the required statement.\smallskip

(3)~In view of Remark \ref{rem1}\,(3), group $GL(B)$ is contractible for spaces of Example \ref{ex2} $B=\ell^p(G,\mu)$, $1<p<\infty$, $c(G)$ or $AP(G)$\footnote{Recall that $AP(G)\cong C(bG)$, where $bG$ is a compact topological group called the {\em Bohr compactification} of $G$.} in case $G$ is infinite and maximally almost periodic (i.e.~finite-dimensional unitary representations separate points of $G$, see, e.g.,~\cite{BrK} for examples of such groups). 
In all these cases, under assumptions of Theorem \ref{te3.5}, we obtain that $H_B^{p,q}(\bar{D})\cong B^m$. In particular, $H_{AP(\mathbb Z^n)}^{p,q}(\bar{D})\cong AP(\mathbb Z^n)^{m}$, where $D\Subset T$ and $r:T'\rightarrow T$ is the covering of Example \ref{ex} ($T\subset\mathbb C^n$ is Stein because it is a relatively complete Reinhardt domain, see, e.g.,~\cite{Shab}).\smallskip

(4)~Similarly, one can reformulate Theorem \ref{thm6} to deal with forms in $\Lambda_{B^*}^{p,q}(X')$ vanishing on $\partial D':=r^{-1}(\partial D)$ in case the dual space $B^*$ of $B$ is a function space on $G$ satisfying conditions of Definition \ref{def1}. This holds, for instance, if $B$ is a reflexive Orlicz space $\ell_{\Phi}$ satisfying assumptions of Example \ref{ex2} or $c(G)$ and $\ell^1(G,\mu)$ spaces of this example. On the other hand, for space $AP(G)$ with $G$ as above the dual $AP(G)^*$ is the space of regular complex Borel measures on $bG$ (the Riesz representation theorem) and therefore to obtain a version of Theorem \ref{thm6} in this case one works with forms in $\Lambda_0^{r,t}(\bar D,E_{X'}^*)$. We leave the corresponding details to the reader.

    \end{remark}

\section{Proof of Theorem \ref{thm2}}

\subsection{Banach-valued differential forms} 
\label{results}

%\smallskip

Let $U\Subset\mathbb C^n$ be a bounded open subset and $B$ a complex Banach space with norm $\|\cdot\|_B$. We fix holomorphic coordinates $z=(z_1,\dots, z_n)$ on $\mathbb C^n$. For tuples $\alpha=(\alpha_1, \dots ,\alpha_p)\in\mathbb N^p$ and $\beta=(\beta_1,\dots,\beta_q)\in\mathbb N^q$,  each consisting of increasing sequences of numbers not exceeding $n$, we set 
\[
|\alpha|:=p,\quad |\beta|:=q\quad\text{and}\quad dz_\alpha\wedge d\bar{z}_\beta:=dz_{\alpha_1}\wedge\cdots\wedge dz_{\alpha_p}\wedge d\bar{z}_{\beta_1}\wedge\cdots\wedge d\bar{z}_{\beta_q}.
\]
As usual, in real coordinates $x_1,\dots, x_{2n}$, $z_j:=x_j+i x_{n+j}$, $1\le j\le n$, on $\mathbb R^{2n}$, partial (Fr\'{e}chet) derivatives $D^\gamma$, $\gamma=(\gamma_1,\dots,\gamma_{2n})\in\mathbb Z_+^{2n}$, of order ${\rm ord}(\gamma):=\gamma_1+\cdots+\gamma_{2n}$ are given by the formulas
\[
D^\gamma:=\frac{\partial^{\gamma_1}}{\partial x_1^{\gamma_1}}\circ\cdots\circ \frac{\partial^{\gamma_{2n}}}{\partial x_{2n}^{\gamma_{2n}}}.
\]

Further, for a $C^k$ $B$-valued $(p,q)$-form $\eta$ on $U$,
\[
\eta(z,\bar{z})=\sum_{|\alpha|=p,|\beta|=q}f_{\alpha,\beta}(z,\bar{z})dz_\alpha\wedge d\bar{z}_\beta,
\]
 and a subset $W\subset U$ we define
\begin{equation}\label{norms}
\begin{array}{l}
\displaystyle
\|\eta\|_{k,W,B}:=\sum_{{\rm ord}(\gamma)\le k, |\alpha|=p,|\beta|=q}\left(\sup_{z\in W}\|D^\gamma f_{\alpha,\beta}(z,\bar{z})\|_B\right)\quad\text{and}\\
\\
\displaystyle
\|\eta\|_{k,W,B}':=\sup_{g\in B^*,\, \|g\|_{B*}\le 1}\left\{\sum_{{\rm ord}(\gamma)\le k, |\alpha|=p,|\beta|=q}\left(\sup_{z\in W}\left|g\bigl(D^\gamma f_{\alpha,\beta}(z,\bar{z})\bigr)\right|\right)\right\}.
\end{array}
\end{equation}
One easily shows that
\begin{equation}\label{equiv}
\frac{1}{c_{p,q,k,n}}\|\eta\|_{k,W,B}\le \|\eta\|_{k,W,B}'\le\|\eta\|_{k,W,B},
\end{equation}
where $c_{p,q,k,n}$ is the cardinality of the set of indices of sums in \eqref{norms}.\smallskip

By $\hat{\Lambda}^{p,q}(W,B)$ we denote the space of $C^\infty$ $B$-valued $(p,q)$-forms $\eta$ on $U$ such that $\|\eta\|_{k,W,B}<\infty$ for all $k \geqslant 0$.
In a standard way one proves that space $\hat{\Lambda}^{p,q}(U,B)$ is complete in the Fr\'{e}chet topology determined by norms $\{\|\cdot\|_{k,U,B}\}_{k\in\mathbb Z_+}$ (cf. \cite[Th.\,7.17]{R}).\smallskip

Now, let us fix a finite family of coordinate charts $(U_j,\varphi_j)$ on $X$ such that $\mathcal U=(U_j)$ forms a finite open cover of an open neighbourhood of $\bar{D}$ and each $\varphi_j$ maps a neighbourhood of $\bar{U}_j$ biholomorphically onto a bounded domain of $\mathbb C^n$. Let $\pi: E\rightarrow X$ be a holomorphic Banach vector bundle with fibre $B$. Using fixed trivializations $\psi_j:E\rightarrow \bar{U}_j\times B$ of $E$ over $\bar{U}_j$ and the holomorphic coordinates on $U_j$ pulled back by $\varphi_j$ from  $\mathbb C^n$, we define %Fr\'{e}chet 
spaces $\hat{\Lambda}^{p,q}(W,E)$, $W\subset U_j\cap D$, of $C^\infty$ $E$-valued $(p,q)$-forms on $U_j\cap D$ as pullbacks of spaces 
$\hat{\Lambda}^{p,q}(\varphi_j(W),B)$.  Seminorms on $\hat{\Lambda}^{p,q}(W,E)$ obtained by pullbacks of seminorms $\|\cdot\|_{k,\varphi_j(W),B}$ are denoted by $\|\cdot\|_{k,W,E}$.
Finally, we equip the space $\Lambda^{p,q}(\bar{D},E):=\Lambda^{p,q}(X,E)|_{\bar{D}}$ of $C^\infty$ $E$-valued forms on $\bar{D}$ with topology  $\tau_{p,q}=\tau_{p,q}(E)$  defined by the sequence of norms $\|\cdot\|_{k,\bar D,E}$, $k\ge 0$, 
\[
\|\eta\|_{k,\bar D,E}:=\sum_j \|\eta|_{U_j\cap D}\|_{k,U_j\cap D,E},\quad \eta\in \Lambda^{p,q}(\bar{D},E).
\] 
Using, e.g., the Hestens extension theorem \cite{He}, one checks easily that $\bigl(\Lambda^{p,q}(\bar{D},E),\tau_{p,q}\bigr)$ is a Fr\'{e}chet space and that topology $\tau_{p,q}$ is independent of the choice of coordinate charts $(U_j,\varphi_j)$ and trivializations $\psi_j$ as above.

If in the above construction we will take pullbacks of norms $\|\cdot\|_{k,\varphi_j(U_j\cap D),B}'$, denoted by $\|\cdot\|_{k,U_j\cap D,E}'$, then due to \eqref{equiv} the sequence of norms $\|\cdot\|_{k,\bar D,E}'$, $k\ge 0$, 
\[
\|\eta\|_{k,\bar D,E}':=\sum_j \|\eta|_{U_j\cap D}\|_{k,U_j\cap D,E}',\quad \eta\in \Lambda^{p,q}(\bar{D},E),
\] 
will produce the same topology on $\Lambda^{p,q}(\bar{D},E)$.

By our definitions, the standard operator 
$$\bar{\partial}:\bigl(\Lambda^{p,q}(\bar D,E),\tau_{p,q}\bigr) \rightarrow \bigl(\Lambda^{p,q+1}(\bar D,E),\tau_{p,q+1}\bigr)$$ 
is continuous. Hence, $\mathcal Z^{p,q}(\bar{D},E) \subset \Lambda^{p,q}(\bar{D},E)$ is a closed subspace.

\subsection{Proof of Theorem \ref{thm2}} \

\smallskip

{\bf A.} First we prove part (1) of the theorem for the trivial bundle $E=X\times B$, where $B$ is a complex Banach space. 
As the required subspace $\mathcal A\subset\mathcal O(X,E)^m$ we will take the space of constant maps $X\rightarrow B^m$ (naturally identified with $B^m$) and as the set $S$ a point of $D$. Then statement (a) of the theorem is obvious. 

Let us show that there exist continuous linear maps
$$
G_B:\Lambda^{p,q}(\bar{D},E) \rightarrow \Lambda^{p,q-1}(\bar{D},E),
$$
$$
H_B:\Lambda^{p,q}(\bar{D},E) \rightarrow \left\{\sum_{i=1}^m f_i\cdot\chi_i:(f_1,\dots,f_m)\in \mathcal A|_{\bar{D}}\right\}\subset \mathcal Z^{p,q}(\bar{D},E)
$$
such that 
\begin{equation}
\label{id_id0}
\omega=\bar{\partial}G_B(\omega)+H_B(\omega)\quad\text{for all}\quad \omega \in \mathcal Z^{p,q}(\bar{D},E).
\end{equation}
Then
\[
\bar{\partial}G_B\oplus H_B:\mathcal Z^{p,q}(\bar{D}, E) \rightarrow \mathcal B^{p,q}(\bar D, E)\oplus \underbrace{(B\otimes_{\mathbb C}\mathcal H^{p,q}(\bar{D}))}_{\cong \mathcal A|_{\bar{D}}}
\]
is an isomorphism of the corresponding Fr\'{e}chet spaces. By the definition its inverse coincides with the operator $L$ which completes the proof of the theorem in this case.

Indeed, for $B=\mathbb C$ existence of the operators $G_{\mathbb C}$ and $H_{\mathbb C}$ is proved in \cite[Ch.~III.1]{FK} (in the terminology of \cite{FK}, $G_{\mathbb C}:=\bar{\partial}^*N$, where $N$ is the ``$\bar{\partial}$-Neumann operator'' and $H_{\mathbb C}$ is the ``orthogonal projection'' onto $\mathcal H^{p,q}(\bar{D}$)). Their continuity in the corresponding Fr\'{e}chet topologies
follows from \cite[Th.~3.1.14]{FK} and the Sobolev embedding theorem.

In the case of the general bundle $E=X \times B$, first we define the required operators on the (algebraic) symmetric tensor product $B \otimes_{\mathbb C} \Lambda^{p,q}(\bar{D})\subset \Lambda^{p,q}(\bar{D},E)$ by the formulas
\[
G_B:={\rm Id}_B\otimes G_{\mathbb C},\qquad H_B:={\rm Id}_B\otimes H_{\mathbb C},
\]
where ${\rm Id}_B: B\rightarrow B$ is the identity operator. If $\omega \in B \otimes \Lambda^{p,q}(\bar{D})$, then due to the continuity of the scalar operators $G_{\mathbb C}$ and $H_{\mathbb C}$ we have, for all $k\ge 0$ and the corresponding norms,
\[
\begin{array}{l}
\displaystyle
\left\|G_B(\omega)\right\|_{k,\bar{D},E}':=\sup_{g\in B^*,\, \|g\|_{B^*}\le 1}\|\bigl(g\otimes {\rm Id}_{\Lambda^{p,q}(\bar{D})}\bigr)\bigl(G_B(\omega)\bigr)\|_{k,\bar{D},X\times\mathbb C}\medskip\\
\displaystyle
=\sup_{g\in B^*,\, \|g\|_{B^*}\le 1}\left\|G_{\mathbb C}\left(\bigl(g\otimes {\rm Id}_{\Lambda^{p,q}(\bar{D})}\bigr)(\omega)\right)\right\|_{k,\bar{D},X\times\mathbb C}\medskip\\
\displaystyle
\le M \cdot\sup_{g\in B^*,\, \|g\|_{B^*}\le 1}\|\bigl(g\otimes {\rm Id}_{\Lambda^{p,q}(\bar{D})}\bigr)(\omega)\|_{k+n+1,\bar{D},X\times\mathbb C}=M\cdot\|\omega\|_{k+n+1,\bar{D},E}'
\end{array}
\]
and, similarly, 
\[
\|H_B(\omega)\|_{k,\bar{D},E}'\le N \cdot \|\omega\|_{k,\bar{D},E}',
\]
where $M$ and $N$ are some constants independent of $\omega$ (but depending on $k,n,D$ and the data in definitions of the above norms).
\begin{remark}\label{rem4.1}
{\rm The shift of index in norms of inequalities for $G_B(\omega)$ results from the fact that in \cite[Th.~3.1.14]{FK} one considers $G_{\mathbb C}$ as a continuous operator between the corresponding Sobolev
spaces $W^k$ and therefore to switch to the case of our norms we must apply the Sobolev embedding theorem. On the other hand, the operator $H_{\mathbb C}$ is defined by the inner product with elements of a basis of $\mathcal H^{p,q}(\bar{D})$ and so its norm as an operator acting in $C^k$ spaces can be estimated directly without involving the Sobolev norms.
}
\end{remark}

The above norm estimates show that linear operators $G_B: \bigl(B \otimes_{\mathbb C} \Lambda^{p,q}(\bar{D}),\tau_{p,q}\bigr)\rightarrow \bigl(B \otimes_{\mathbb C} \Lambda^{p,q}(\bar{D}),\tau_{p,q-1}\bigr)$ and
$H_B: \bigl(B \otimes_{\mathbb C} \Lambda^{p,q}(\bar{D}),\tau_{p,q}\bigr)\rightarrow (B \otimes_{\mathbb C} \mathcal H^{p,q}(\bar{D}) ,\tau_{p,q})$ are uniformly continuous.
Since  $B \otimes \Lambda^{p,q}(\bar{D})$ is dense in $\bigl(\Lambda^{p,q}(\bar{D},E),\tau\bigr)$ (this can be easily seen using, e.g., ~approximation of local coefficients of forms in $\Lambda^{p,q}(\bar{D},E)$ by their Taylor polynomials and then patching these approximations together by suitable partitions of unity), the latter implies that $G_B$ and $H_B$ can be extended to continuous operators on $\bigl(\Lambda^{p,q}(\bar{D},E),\tau_{p,q}\bigr)$ with ranges in  $\Lambda^{p,q-1}(\bar{D},E)$ and $B \otimes_{\mathbb C} \mathcal H^{p,q}(\bar{D})$, respectively. We retain the same symbols for the extended operators.

Let us show that so defined operators satisfy identity \eqref{id_id0}. 

In fact, for each $g\in B^*$ the linear map $g\otimes {\rm Id}_{\Lambda^{p,q}(\bar{D})}: B\otimes_{\mathbb C}\Lambda^{p,q}(\bar{D})\rightarrow \Lambda^{p,q}(\bar{D})$ is uniformly continuous in the corresponding Fr\'{e}chet topologies and therefore is extended to  a linear continuous map 
\begin{equation}
\label{g_op}
\hat g_{p,q}: \Lambda^{p,q}(\bar{D},E)\rightarrow \Lambda^{p,q}(\bar{D}). 
\end{equation}
Clearly, 
\[
\hat g_{p,q}\circ\bar{\partial}=\bar{\partial}\circ\hat g_{p,q-1},\quad \hat g_{p,q-1}\circ G_B=G_{\mathbb C}\circ\hat g_{p,q}\quad\text{and}\quad \hat g_{p,q}\circ H_B=H_{\mathbb C}\circ\hat g_{p,q}.
\]
In particular, for $\omega\in \mathcal Z^{p,q}(\bar{D},E)$ we have $\hat g_{p,q}(\omega)\in\mathcal Z^{p,q}(\bar{D})$; hence, due to the previous identities and since \eqref{id_id0} is valid for $B=\mathbb C$, 
\[
\hat g_{p,q}\bigl(\bar{\partial}G_B(\omega)+H_B(\omega)\bigr)=\bar{\partial}G_{\mathbb C}(\hat g_{p,q}(\omega))+H_{\mathbb C}(\hat g_{p,q}(\omega))=\hat g_{p,q}(\omega)\quad\text{for all}\quad g\in B^*.
\] 
It is easily seen that the family of linear maps $\{\hat g_{p,q}\, :\, g\in B^*\}$ separates the points of $\Lambda^{p,q}(\bar{D},E)$. Therefore the latter implies that $\bar{\partial}G_B(\omega)+H_B(\omega)=\omega$ for all $\omega\in \mathcal Z^{p,q}(\bar{D},E)$, as required.

\smallskip

\textbf{B.~}Now, we consider the case of an arbitrary holomorphic Banach vector bundle $E \in \Sigma_0(X)$. By the definition, there exists a holomorphic Banach vector bundle $E_1\rightarrow X$ such that $E_2:=E\oplus E_1$ is holomorphically trivial Banach vector bundle with a fibre $B_2$. By $i:E\rightarrow E_2$ and $r:E_2\rightarrow E$, $r\circ i:= {\rm Id}_E$, we denote the corresponding bundle homomorphisms. In a natural way, they induce continuous linear maps of the corresponding Fr\'{e}chet spaces:
$$
\hat{i}^{p,q}:\bigl(\Lambda^{p,q}(\bar{D},E),\tau_{p,q}(E)\bigr) \rightarrow \bigl(\Lambda^{p,q}(\bar{D},E_2),\tau_{p,q}(E_2)\bigr),
$$
$$
\hat{r}^{p,q}: \bigl(\Lambda^{p,q}(\bar{D},E_2),\tau_{p,q}(E_2)\bigr) \rightarrow \bigl(\Lambda^{p,q}(\bar{D},E),\tau_{p,q}(E)\bigr)
$$
such that $\hat r^{p,q}\circ\hat i^{p,q}={\rm Id}_{\Lambda^{p,q}(\bar{D},E)}$. Moreover, $\hat i^{p,q}$ and $\hat r^{p,q}$ commute with the corresponding $\bar{\partial}$ operators and therefore 
$\hat i^{p,q}$ embeds $\mathcal Z^{p,q}(\bar{D},E)$ as a closed subspace into $\mathcal Z^{p,q}(\bar{D},E_2)$ and $\hat r^{p,q}$ maps $\mathcal Z^{p,q}(\bar{D},E_2)$
surjectively onto $\mathcal Z^{p,q}(\bar{D},E)$.

Next, we define continuous linear operators
$$
G_E := \hat{r}^{p,q-1} \circ G_{B_2} \circ \hat{i}^{p,q}:\Lambda^{p,q}(\bar{D},E)\rightarrow \Lambda^{p,q-1}(\bar{D},E),
$$
\[
H_E := \hat{r}^{p,q} \circ H_{B_2} \circ \hat{i}^{p,q}:\Lambda^{p,q}(\bar{D},E)\rightarrow \mathcal Z^{p,q}(\bar{D}, E),
\]
where $G_{B_2}$ and $H_{B_2}$ are operators constructed in part {\bf A} for the trivial bundle $E_2:=X\times B_2$. Due to identity \eqref{id_id0} for these operators we have
\begin{equation}\label{e4.6}
\omega=\bar{\partial}G_E(\omega)+H_E(\omega)\quad\text{for all}\quad \omega \in \mathcal Z^{p,q}(\bar{D},E).
\end{equation}
This implies (since $H_{B_2}$ maps $\bar{\partial}$-exact forms to $0$)
\[
H_E(\omega)=H_E^2(\omega)\quad\text{for all}\quad \omega \in \mathcal Z^{p,q}(\bar{D},E).
\]
Thus $H_E(\mathcal Z^{p,q}(\bar{D},E))$ is a closed complemented subspace of $\mathcal Z^{p,q}(\bar{D},E)$ and \eqref{e4.6} shows that $\mathcal Z^{p,q}(\bar{D},E)=\mathcal B^{p,q}(\bar{D},E)\oplus H_E(\mathcal Z^{p,q}(\bar{D},E))$.

Further, since each $\eta\in  B_2\otimes_{\mathbb C}\Lambda^{p,q}(\bar{D})$ is uniquely presented as $\eta=\sum_{i=1}^m c_i(\eta)\cdot\chi_i$ for some $c_i(\eta)\in B_2$, by the open mapping theorem the correspondence $\eta\mapsto (c_1(\eta),\dots,c_m(\eta))$ determines an isomorphism of the Fr\'{e}chet spaces $c: \bigl(B_2\otimes_{\mathbb C}\Lambda^{p,q}(\bar{D}), \tau_{p,q}(E_2)\bigr)\rightarrow B_2^m$.
In what follows we regard $B_2$ as the subset of $\mathcal O(X,E_2)$ consisting of constant sections. Also, we equip the space $\mathcal O(X,E)^m$ of holomorphic sections of $\oplus^m E$ with topology of uniform convergence on compact subsets of $X$.

We have the following sequence of continuous linear maps
\begin{equation}\label{e4.7}
\mathcal O(X,E)^m\stackrel{t}{\longrightarrow}\Lambda^{p,q}(\bar{D},E)\stackrel{H_{B_2}\circ\hat i^{p,q}}{\longrightarrow} B_2\otimes_{\mathbb C}\Lambda^{p,q}(\bar{D})\stackrel{c}{\longrightarrow} B_2^m\stackrel{\hat r}{\longrightarrow} \mathcal O(X,E)^m,
\end{equation}
where $t(f_1,\dots,f_m):=\sum_{i=1}^m f_i|_{\bar{D}}\cdot\chi_i$, $(f_1,\dots, f_m)\in \mathcal O(X,E)^m$, and $\hat r:=\oplus^m(\hat r^{0,0}|_{B_2})$.\medskip

Let us define the required space $\mathcal A\subset \mathcal O(X,E)^m$ of the theorem as the image of $\mathcal Z^{p,q}(\bar{D},E)$ under the map $P:=\hat r\circ c\circ H_{B_2}\circ\hat i^{p,q}$.\medskip

By our definition, $t\circ P=H_{E}$ on $\mathcal Z^{p,q}(\bar{D},E)$, and since $H_E$ is the identity map on $H_E(\mathcal Z^{p,q}(\bar{D},E))$ and zero on $\mathcal B^{p,q}(\bar{D},E)$, the subspace $\mathcal A\subset \mathcal O(X,E)^m$ is closed and $P: H_E(\mathcal Z^{p,q}(\bar{D},E))\to\mathcal A$ is an isomorphism with inverse $t|_{\mathcal A}$. Therefore, $P\circ t: \mathcal O(X,E)^m\rightarrow \mathcal O(X,E)^m$ is a projection onto $\mathcal A$, that is, $\mathcal A\subset \mathcal O(X,E)^m$ is a complemented subspace. Also, the map $L:={\rm Id}_{\mathcal B^{p,q}(\bar D,E)}\oplus t|_{\mathcal A}: \mathcal B(\bar{D},E)\oplus\mathcal A\rightarrow \mathcal Z^{p,q}(\bar{D},E)$ is an isomorphism of the Fr\'{e}chet spaces.  Note that ${\rm Ker}\, (P\circ t)$ consists of all $(f_1,\dots, f_m)\in\mathcal O(X,E)^m$ such that $t(f_1,\dots, f_m)\in\mathcal B^{p,q}(\bar D,E)$.
\smallskip

Now, to define the required set $S\subset\bar D$ of the theorem and to prove statement (a) let us prove, first, the following result.
\begin{lemma}\label{lem4.2}
The restriction map $R_{\bar D}:\mathcal O(X,E)^m\rightarrow C(\bar{D},E)^m$ to $\bar{D}$ maps $\mathcal A$ isomorphically onto a closed %complemented 
subspace of the space $A(\bar{D},E)^m$, where $A(\bar D,E)$ is the closure in $C(\bar D,E)$ of the trace space $\mathcal O(X,E)|_{\bar D}$. 
\end{lemma}
\begin{proof}
Indeed, map $t$ in \eqref{e4.7} can be factorized as $t=\tilde t\circ R_{\bar D}$ for a continuous linear map  
\[
\tilde t: A(\bar{D},E)^m\rightarrow \bar{\Lambda}^{p,q}(\bar{D},E),\quad
\tilde t(g_1,\dots,g_m):=\sum_{i=1}^m g_i\cdot\chi_i,\quad (g_1,\dots, g_m)\in A(\bar D,E)^m,
\]
where $\bar{\Lambda}^{p,q}(\bar{D},E)$ is the completion of the normed space $\bigl(\Lambda^{p,q}(\bar{D},E), \|\cdot\|_{0,\bar{D},E}\bigr)$.

\noindent Also, by our construction, see part {\bf A} above, 
map 
\[
H_{B_2}: \bigl(\Lambda^{p,q}(\bar D, E_2),\|\cdot\|_{0,\bar{D},E_2}\bigr)\rightarrow \bigl(B_2\otimes_{\mathbb C}\mathcal H^{p,q}(\bar D), \|\cdot\|_{0,\bar{D},E_2}\bigr)
\] 
is continuous and, hence, admits a continuous extension
\[
\bar{H}_{B_2}: \bar{\Lambda}^{p,q}(\bar D, E_2)\rightarrow \bigl(B_2\otimes_{\mathbb C}\mathcal H^{p,q}(\bar D), \|\cdot\|_{0,\bar{D},E_2}\bigr);
\]
here $\bar{\Lambda}^{p,q}(\bar D, E_2)$ is the completion of the space $\bigl(\Lambda^{p,q}(\bar D, E_2),\|\cdot\|_{0,\bar{D},E_2}\bigr)$.

Then the composite map  $c\circ \bar{H}_{B_2}\circ\hat i^{p,q}\circ\tilde t: A(\bar{D},E)^m\rightarrow B_2^m$ is continuous with respect to the corresponding norms $\|\cdot\|_{0,\bar{D},E^m}$ and $\|\cdot\|_{0,\bar{D},E_2^m}$ on $A(\bar{D},E)^m$ and $B_2^m$. (Here for a Banach vector bundle $V\rightarrow X$ we set $V^m:=\oplus^m V$.) Note that topologies defined by these norms coincide with topology of uniform convergence for $A(\bar{D},E)^m$ and topology defined by the Banach norm for $B_2^m$. Therefore, if $\{F_k\}_{k\in\mathbb N}\subset\mathcal A|_{\bar D}$ is a Cauchy sequence, then the sequence $\{b_k:=(c\circ H_{B_2}\circ\hat i^{p,q}\circ\tilde t)(F_k)\}_{k\in\mathbb N}$ converges in $B_2^m$, and, hence, $\{(\hat r(b_k)\}_{k\in\mathbb N}$ converges in $\mathcal O(X,E)^m$ (in topology of uniform convergence on compact subsets of $X$). Since $F_k=(R_{\bar D}\circ \hat r)(b_k)$ for all $k$ and $\mathcal A\subset\mathcal O(X,E)^m$ is closed,  $\{F_k\}_{k\in\mathbb N}$ converges in $A(\bar{D},E)^m$ to an element of $\mathcal A|_{\bar D}$, as required. Thus, by the open mapping theorem, $R_{\bar D}:\mathcal A\rightarrow\mathcal A|_{\bar D}$ is an isomorphism of the corresponding Fr\'{e}chet spaces.  
%Finally, by the definition of the space $A(\bar{D},E)$, the map $R_{\bar D}\circ \bar{H}_{B_2}\circ\hat i^{p,q}\circ \tilde t$ is  a continuous projection onto $\mathcal A|_{\bar D}$; hence $\mathcal A|_{\bar D}$ is a complemented subspace of $A(\bar{D},E)^m$.
\end{proof}

Let $D'\supset\bar D$ be a relatively compact subdomain of $X$. We equip the space $C(\bar{D'},E)$ with a norm $\|\cdot\|_{0,\bar{D'},E}$ defined similarly to $\|\cdot\|_{0,\bar{D},E}$ (see subsection~4.1). Topology defined by this norm is topology of uniform convergence on $\bar{D'}$, and $\bigl(C(\bar{D'},E),\|\cdot\|_{0,\bar{D'},E}\bigr)$ is a Banach space. We define $A(\bar{D'},E)$ to be the closure in $C(\bar{D'},E)$ of the trace space
$\mathcal O(X,E)|_{\bar{D'}}$. We have the following sequence of continuous linear maps (induced by subsequent restrictions $X$ to $\bar{D'}$ and $\bar{D'}$ to $\bar{D}$)
\[
\mathcal O(X,E)^m\stackrel{R_{\bar{D'}}}{\longrightarrow}A(\bar{D'},E)^m\stackrel{R_{\bar{D}}^{\bar{D'}}}{\longrightarrow}A(\bar D,E)^m
\]
such that $R_{\bar D}=R_{\bar{D}}^{\bar{D'}}\circ R_{\bar{D'}}$.

As a straightforward corollary of Lemma \ref{lem4.2} we obtain
\begin{lemma}\label{le4.3}
$\mathcal A|_{\bar{D'}}$ is a closed subspace of $A(\bar{D'},E)^m$ and $R_{\bar{D}}^{\bar{D'}}$ maps $\mathcal A|_{\bar{D'}}$ isomorphically onto $\mathcal A|_{\bar{D}}$.
\end{lemma}

In particular, this lemma implies that there exists a constant $C>0$ such that
\begin{equation}\label{e4.8}
\| R_{\bar{D}}^{\bar{D'}}(v)\|_{0,\bar D,E^m}\ge C \|v\|_{0,\bar{D'},E^m}\quad\text{for all}\quad v\in \mathcal A|_{\bar{D'}}.
\end{equation}

Let us fix a complete (smooth) Hermitian metric on $X$ and with its help define the path metric $d: X\times X\rightarrow\mathbb R_+$. For a fixed $\varepsilon>0$ by $S_\varepsilon\subset\bar D$ we denote an $\varepsilon$-net in $\bar D$ with respect to the metric $d$.
\begin{proposition}\label{prop4.4}
For a sufficiently small $\varepsilon$ the restriction map $R_{S_\varepsilon}:A(\bar D,E)\rightarrow C(S_\varepsilon,E)^m$ to $S_\varepsilon$ maps $\mathcal A|_{\bar D}$ isomorphically onto a closed 
%complemented 
subspace of $C(S_\varepsilon, E)^m$.
\end{proposition}
\begin{proof}
If $v\in A(\bar{D'},E)^m\, (=A(\bar{D'}, E^m)$, then according to the Cauchy estimates for derivatives of bounded holomorphic functions we have for a constant $C'>0$ depending on $D$, $D'$ and definitions of the corresponding norms
\[
\|\partial v\|_{1,\bar{D},E^m}\le C'\|v\|_{0,\bar{D'},E^m}.
\] 
This, the definition of the metric $d$ and the intermediate-value inequality imply that there exist a constant $C''>0$ (independent of $v$) such that for all $\varepsilon>0$ and $x_1,x_2\in\bar{D}$ satisfying $d(x_1,x_2)\le\varepsilon$,
\begin{equation}\label{e4.9}
\max_{i=1,2}\|v(x_1)-v(x_2)\|_{0,\{x_i\},E^m}\le C''\cdot\varepsilon\cdot \|v\|_{0,\bar{D'},E^m}.
\end{equation}
Let us choose $\varepsilon$ so that
\[
0<\varepsilon\le\frac{C}{2C''},
\]
where $C$ is defined in \eqref{e4.8}. If $v\in \mathcal A$ is such that $\|v|_{\bar D}\|_{0,\bar{D},E}=1$, then according to \eqref{e4.8}, $\|v|_{\bar{D'}}\|_{0,\bar{D'},E}\le\frac{1}{C}$, and \eqref{e4.9} implies that
$\|v|_{S_{\varepsilon}}\|_{0,S_{\varepsilon},E^m}\ge\frac 12$. Hence, we have
\begin{equation}\label{e4.10}
\|R_{S_\varepsilon}(v)\|_{0,S_{\varepsilon},E^m}\ge\frac 12\|v\|_{0,\bar D,E^m}\quad\text{for all}\quad v\in\mathcal A|_{\bar{D}}.
\end{equation}
This shows that $R_{S_\varepsilon}$ maps $\mathcal A|_{\bar D}$ isomorphically onto a closed subspace of $C(S_\varepsilon, E)^m$. %Let us prove that this subspace is complemented.
\end{proof}

Taking $S:=S_\varepsilon$ in statement (a) of the theorem with $\varepsilon$ as in Proposition \ref{prop4.4} we obtain the required result; this
completes the proof of part (1) of the theorem.

\medskip

(2) Suppose $E=f^*E'$, where $f:X\rightarrow Y$ is a holomorphic map into a Stein manifold $Y$ and $E'$ is a holomorphic Banach vector bundle on $Y$ with fibre $B$ such that the group $GL(B)$ is contractible. The latter implies that $E'$ is isomorphic to the trivial bundle $Y \times B$ in the category of topological Banach vector bundles. In turn, since $Y$ is Stein, the Oka principle for holomorphic Banach vector bundles, see \cite{Bun}, implies that $E'$ is holomorphically isomorphic to $Y \times B$ as well, and so $E$ is holomorphically isomorphic to $X \times B$. Thus, the required result follows from part (1) of the theorem applied to the trivial bundle $X \times B$.

The proof of the theorem is complete.
%=========================
\section{Proof of Theorem \ref{thm6}}

The isomorphism in (1) is 
induced by the map associating to each $\xi \in \mathcal Z_0^{n-p,n-q}(\bar{D},E^*)$ a linear functional
$$
\mathcal Z^{p,q}(\bar{D},E) \ni \theta \mapsto J_E(\theta,\xi),
$$
where 
$$
J_E:\Lambda^{p,q}(\bar{D},E) \times \Lambda^{n-p,n-q}(\bar{D},E^*) \rightarrow \mathbb C
$$ 
is a continuous bilinear form with respect to the product topology $\tau_{p,q}(E)\times\tau_{n-p,n-q}(E^*)$ on $\Lambda^{p,q}(\bar{D},E) \times \Lambda^{n-p,n-q}(\bar{D},E^*)$, see subsection~4.1, defined as follows.% (We define $J_E$ for an arbitrary $E$, albeit we will use it for trivial bundles only.)

Let  $V\rightarrow X$  be a holomorphic Banach vector bundle.
Consider a continuous bundle homomorphism $$\tr_E(V):  E\otimes E^*\otimes V \rightarrow V$$ sending a vector $e_x\otimes e_x^*\otimes v_x$ in the fibre of $E\otimes E^*\otimes V$ over $x\in X$ to the vector $e_x^*(e_x) \cdot v_x$ in the fibre of $V$ over $x$ (here $e_x$, $e_x^*$ and $v_x$ are vectors in fibres of $E$, $E^*$ and $V$ over $x$) and then extended by linearity.

To define $J_E$ we take
\[
V_{p,q,r,s}:=\biggl(\bigl(\wedge^p T^*\bigr)\wedge  (\wedge^q \overline{T}^*) \biggr) \otimes \biggl(\bigl(\wedge^r T^* \bigr)\wedge (\wedge^s \overline{T}^*)\biggr).
\]
By definition, forms in $\Lambda^{p,q}(\bar D, E)$ are $C^\infty$ sections over $\bar D$ of bundle $E\otimes\bigl(\wedge^p T^*\bigr)\wedge  (\wedge^q \overline{T}^*)$, forms in $\Lambda^{r,s}(\bar D, E^*)$ are $C^\infty$ sections over $\bar D$ of bundle $E^* \otimes\bigl(\wedge^r T^* \bigr)\wedge (\wedge^s \overline{T}^*)$.
Therefore, if
$\theta \in  \Lambda^{p,q}(\bar D, E)$ and $\xi \in \Lambda^{r,s}(\bar D, E^*)$, then $\theta \otimes \xi$ is a $C^\infty$ section over $\bar D$ of  bundle 
\begin{equation*}
\biggl( E\otimes\bigl(\wedge^p T^*\bigr) \wedge(\wedge^q \overline{T}^*) \biggr)\otimes \biggl( E^*\otimes\bigl(\wedge^r T^* \bigr)\wedge  (\wedge^s \overline{T})^*\biggr) \cong E\otimes E^*\otimes V.
\end{equation*}
In turn, $\tr_{E}(V_{p,q,r,s})(\theta \otimes \xi)$ is a $C^\infty$ section over $\bar D$ of bundle $V_{p,q,r,s}$. Let 
$\Lambda_{p,q,r,s}$ be the canonical quotient homomorphism of bundles $V \rightarrow \bigl(\wedge^{p+r} T^*\bigr)  \wedge \bigl(\wedge^{q+s} \overline{T}^*\bigr)$ (obtaining by replacing $\otimes$ by $\wedge$ in the definition of $V$). Assuming that $r=n-p$, $s=n-q$, we set
\begin{equation}\label{je}
J_E(\theta,\xi):=\int_{D}\Lambda_{p,q,n-p,n-q}\biggl(\tr_{E}(V_{p,q,n-p,n-q})(\theta \otimes \xi)\biggr)
\end{equation}
(by definition, the integrand is in $\Lambda^{n,n}(\bar{D})$, and so the integral is well defined). The construction of $J_E$ and the definition of
norms $\|\cdot\|_{0,\bar{D},E}$ and  $\|\cdot\|_{0,\bar{D},E^*}$ given in subsection~4.1 imply immediately
\begin{lemma}
\label{bddlem}
There is a constant $C>0$ such that 
\[
|J_E(\theta,\xi)| \leqslant C\cdot\|\theta\|_{0,\bar{D},E}\cdot \|\xi\|_{0,\bar{D},E^*}\quad\text{for all}\quad \theta \in \Lambda^{p,q}(\bar D,E),\ \xi \in \Lambda^{n-p,n-q}(\bar{D},E^*).
\]
In particular, $J_E:\bigl(\Lambda^{p,q}(\bar{D},E) \times \Lambda^{n-p,n-q}(\bar{D},E^*), \tau_{p,q}(E)\times\tau_{n-p,n-q}(E^*)\bigr) \rightarrow \mathbb C$ is a continuous bilinear form.
\end{lemma}

We are in position to prove the theorem.

\smallskip

\textbf{I.}~First, we prove the result for the case of the trivial bundle $E=X \times B$, where $B$ is a complex Banach space. 
Let us prove (1). 
\begin{lemma}\label{stokes}
$J_E=0$ on $\bigl(\mathcal Z^{p,q}(\bar{D},E)\times  \mathcal B_0^{n-p,n-q}(\bar{D},E^*)\bigr)\bigcup\bigl(\mathcal B^{p,q}(\bar{D},E)\times \mathcal Z_0^{n-p,n-q}(\bar{D},E^*)\bigr)$.
\end{lemma}

\begin{proof}
For $E=X\times\mathbb C$ (the scalar case) the required result is proved in \cite[Ch.V.1]{FK}. The proof in the general case repeats word-for-word the previous one and is based on the following identities, the first one valid for all $\phi \in \Lambda^{p,q}(\bar{D},E), \psi \in \Lambda^{r,s}(\bar{D},E^*)$, and the second one for all
$\phi \in \Lambda^{p,q}(\bar{D},E), \psi \in \Lambda^{n-p,s}(\bar{D},E^*)$  with $q+s=n-1$:
$$
\begin{array}{l}
\displaystyle
\bar{\partial} \Lambda_{p,q,r,s}\biggl(\tr_{E}(V_{p,q,r,s})(\phi \otimes \psi)\biggr)=
\Lambda_{p,q+1,r,s,}\biggl(\tr_{E}(V_{p,q+1,r,s})(\bar{\partial}\phi \otimes \psi)\biggr)\medskip\\
\displaystyle
+(-1)^{p+q}\Lambda_{p,q,r,s+1}\biggl(\tr_{E}(V_{p,q,r,s+1})(\phi \otimes \bar{\partial}\psi)\biggr),
\end{array}
$$
$$
\int_{\bar{D}} \bar{\partial} \Lambda_{p,q,n-p,s}\biggl(\tr_{E}(V_{p,q,n-p,s})(\phi \otimes \psi)\biggr)
=\int_{\partial D}
 \Lambda_{p,q,n-p,s}\biggl(\tr_{E}(V_{p,q,n-p,s})(\phi \otimes \psi)\biggr).
$$
(The first identity is easily verified in local coordinates. The second one is the Stokes theorem.)
\end{proof}

Lemmas \ref{stokes} and \ref{bddlem} and the fact that $\mathcal B^{p,q}(\bar{D},E)\subset \mathcal Z^{p,q}(\bar{D},E)$ is a closed subspace imply that $J_E$ descends to a bilinear form 
\begin{equation}
\label{L2}
\mathcal J_E: H^{p,q}(\bar{D},E) \times H_0^{n-p,n-q}(\bar{D},E^*) \rightarrow \mathbb C
\end{equation}
such that $S_E(h_0):=\mathcal J_E(\cdot,h_0)\in (H^{p,q}(\bar{D},E))^*$ for each $h_0\in H_0^{n-p,n-q}(\bar{D},E^*)$.

Let us prove that the linear map 
\begin{equation}
\label{cohmap}
S_E:H_0^{n-p,n-q}(\bar{D},E^*) \rightarrow (H^{p,q}(\bar{D},E))^*
\end{equation}
is injective and surjective. Along the lines of the proof, we will show that $\mathcal B_0^{n-p,n-q}(\bar{D},E^*)$ is a closed subspace of $\mathcal Z_0^{n-p,n-q}(\bar{D},E^*)$, which will prove assertion (1) in this case.

\smallskip

Thus, we must prove:

\smallskip

a) ({\em surjectivity}) given an element $F \in (H^{p,q}(\bar{D},E))^*$, there exists $\xi \in \mathcal Z_0^{n-p,n-q}(\bar{D},E^*)$ such that $J_E(\theta,\xi)=F([\theta])$ for all $\theta \in \mathcal Z^{p,q}(\bar{D},E)$; here $[\theta]\in H^{p,q}(\bar{D},E)$ denotes the cohomology class of $\theta$.

\smallskip

b) ({\em injectivity}) if $J_E(\theta,\xi)=0$ for all $\theta \in \mathcal Z^{p,q}(\bar{D},E)$, then $\xi \in \mathcal B_0^{n-p,n-q}(\bar{D},E^*)$.

\smallskip

First, let us prove a). Recall that we  fix forms
$\{\gamma_i\}_{i=1}^m \subset \mathcal Z_0^{n-p,n-q}(\bar{D})$ such that $\int_D \chi_i \wedge \gamma_j=\delta_{ij}$ - the Kronecker delta; here
$\{\chi_i\}_{i=1}^m$ is the basis of $\mathcal H^{p,q}(\bar{D})$.

Due to \eqref{id_id0} of subsection~4.2, each form $\theta \in \mathcal Z^{p,q}(\bar{D},E)$ can be uniquely presented as $\theta=\bar{\partial} G_B(\theta)+H_B(\theta)$, where $H_B(\theta)=\sum_{i=1}^m b_i \cdot \chi_i$ and all $b_i \in B$. Therefore the correspondence $[\theta]\mapsto (b_i)_{i=1}^m$ determines an isomorphism of Fr\'{e}chet spaces $H^{p,q}(\bar{D},E)\cong B^m$. Under this isomorphism,
$(H^{p,q}(\bar{D},E))^* \cong (B^*)^m$ and so each $F\in H^{p,q}(\bar{D},E))^*$ has a form $F=(b_i^*)_{i=1}^m\in (B^*)^m$ for some $b_i^*\in B^*$, and
$F([\theta]):=\sum_{i=1}^m b_i^*(b_i)$.

Now, we set $\xi:=\sum_{i=1}^m b_i^*\cdot\gamma_i \in \mathcal Z_0^{n-p,n-q}(\bar{D},E^*)$. Then, by the definition of $J_E$ we have
\[
J_E(\theta,\xi)=J_E\left(\sum_{i=1}^m b_i\cdot\chi_i,\sum_{i=1}^m b_i^* \cdot \gamma_i\right)=\sum_{i,j=1}^mb_i^*(b_j)\int_D \chi_i \wedge \gamma_j=
\sum_{i=1}^m b_i^*(b_i)=F([\theta]),
\]
as required.\smallskip

Next, let us prove b). We construct a continuous linear operator 
\begin{equation}\label{qe}
Q_{E}:\Lambda^{n-p,n-q}(\bar{D},E^*) \rightarrow \Lambda_0^{n-p,n-q-1}(\bar{D},E^*)
\end{equation}
 such that if $\xi \in \mathcal Z_0^{n-p,n-q}(\bar{D},E^*)$ and $J_E(\cdot,\xi)=0$, then %$Q_{E}\,\xi \in \Lambda_0^{n-p,n-q-1}(\bar{D},E^*)$ and 
$\xi=\bar{\partial}(Q_{E}\, \xi)$. 
Clearly, existence of such an operator would imply b) and, hence, show that $\mathcal B_0^{n-p,n-q}(\bar{D},E^*)$ is a closed subspace of $\mathcal Z_0^{n-p,n-q}(\bar{D},E^*)$.

In case $E=X \times \mathbb C$ the required operator was constructed in \cite[Ch.V.1]{FK}: $$Q_{X \times \mathbb C}\,\psi:=\rho(-\bar{\partial}\alpha+\theta); 
$$
here $\rho$ is the defining function of $D$ and $\alpha \in \Lambda^{n-p,n-q-2}(\bar{D},E^*),\, \theta \in \Lambda^{n-p,n-q-1}(\bar{D},E^*)$
are uniquely determined by the formula
$$\ast\overline{\bar{\partial} N(\ast \bar{\psi})} =: \bar{\partial}\rho \wedge \alpha +\rho \theta~(=\bar{\partial}(\rho\alpha)+\rho(-\bar{\partial}(\rho\alpha)+\theta)).$$ 
Here $\ast$ is the Hodge star operator and $N$ is the ``$\bar{\partial}$-Neumann operator'' in the terminology of \cite[Ch.V.1]{FK}; the continuity of $Q_{X \times \mathbb C}$ in the Fr\'{e}chet topology on $\Lambda^{n-p,n-q}(\bar{D})$ follows from \cite[Th.~3.1.14]{FK} and the Sobolev embedding theorem. 

In the general case, we define $Q_{E}$ using $Q_{X \times \mathbb C}$ similarly to how it was done for operators $G_B$, $H_B$ in part {\bf A} of the proof of Theorem \ref{thm2}, cf. subsection~4.2: first, we define $Q_{E}:=\Id_{B^*} \otimes Q_{X \times \mathbb C}$ on the tensor product $B^* \otimes \Lambda^{n-p,n-q}(\bar{D})$. Then, using the facts that $B^* \otimes \Lambda^{n-p,n-q}(\bar{D})$ is dense in $\Lambda^{n-p,n-q}(\bar{D},E^*)$ and  that in virtue of continuity of $Q_{X \times \mathbb C}$ operator $Q_{E}$ is bounded with respect to Fr\'{e}chet seminorms $\|\cdot\|'_{k,\bar{D},E^*}$ on $\Lambda^{n-p,n-q}(\bar{D},E^*)$, we extend $Q_{E}$ by continuity to $\Lambda^{n-p,n-q}(\bar{D},E^*)$.
Now, we prove that the constructed operator $Q_{E}$ possesses the required properties. 

Indeed, by definition, inclusion $Q_{E}\xi \in \Lambda_0^{n-p,n-q-1}(\bar{D},E^*)$ is equivalent to identity $(Q_{E}\xi)|_{\partial D}=0$. It is verified by applying to $Q_{E}\xi$ ``scalarization operators'' $\hat{g}_{n-p,n-q}:\Lambda^{n-p,n-q}(\bar{D},E^*) \rightarrow \Lambda^{n-p,n-q}(\bar{D})$ (cf.~\eqref{g_op} with $g$ viewed as an element of $B^{**}$), and using that $\hat{g}_{n-p,n-q}(Q_{E}\xi)=
Q_{X\times\mathbb C}\, \hat{g}_{n-p,n-q}(\xi)$ and the latter vanishes on $\partial D$. Identity $\xi=\bar{\partial}(Q_{E} \xi)$ for
$\xi$ satisfying $J_E(\cdot,\xi)=0$
is also verified by this method. This completes the proof of b).

\smallskip

To finish the proof of assertion (1) it remains to show that $S_E$ and its inverse  are continuous (see \eqref{cohmap}). Indeed, continuity of $S_E$ follows from Lemma \ref{bddlem} and the fact that $\mathcal B_0^{n-p,n-q}(\bar{D},E^*)$ is a closed subspace of $\mathcal Z_0^{n-p,n-q}(\bar{D},E^*)$. Continuity of the map inverse to $S_E$ follows from the open mapping theorem for linear continuous maps between Fr\'{e}chet spaces.

\smallskip

Let us prove part (2) of the theorem in the case of trivial bundles. According to part {\bf A} of the proof of Theorem \ref{thm2}, $\mathcal A=B^m$, i.e., consists of constant sections in
$\mathcal O(X,E)^m$. Hence, $\mathcal B:=(B^*)^m$ consists of constant sections in $\mathcal O(X,E^*)^m$. As the required set $S^*$ we take a point in $\bar D$; then the statement $\mathcal B\cong\mathcal B|_{S^*}$ is obvious.
The fact that the linear map $M:\mathcal B \rightarrow H_0^{n-p,n-q}(\bar{D},E^*)$,
$$
M(h_1,\dots,h_m):=\left[\sum_{i=1}^m h_i|_{\bar D} \cdot \gamma_i\right], \quad (h_1,\dots,h_m) \in \mathcal B,
$$
is an isomorphism of Fr\'{e}chet spaces (i.e., statement (2)\,(b)) follows from the arguments presented in the proof of a) and b) above.

This completes the proof of the theorem for trivial bundles.

\medskip

\textbf{II.~}Now we consider the case of an arbitrary holomorphic Banach vector bundle $E \in \Sigma_0(X)$. 
Recall that by the definition of class $\Sigma_0(X)$ there exists a holomorphic Banach vector bundle $E_1$ on $X$ such that
the Whitney sum $E_2:=E \oplus E_1$ is holomorphically trivial, i.e. $E_2= X \times B_2$ for a complex Banach space $B_2$. 
We have the corresponding embedding and quotient homomorphisms of bundles
$$
i:E \rightarrow E_2 \quad\text{and}\quad r:E_2 \rightarrow E\quad \text{such that}\quad r\circ i={\rm Id}_{E}.
$$
In turn, $E_2^*=E^* \oplus E_1^*$ and we have the adjoint homomorphisms
$$
i^*:E_2^* \rightarrow E^*\quad\text{and}\quad r^*:E^* \rightarrow E_2^*\quad\text{such that}\quad i^*\circ r^*={\rm Id}_{E^*}.
$$
The above homomorphisms induce continuous linear maps between the corresponding Fr\'{e}chet spaces of forms
$$\hat{i}^{s,t}:\Lambda^{s,t}(\bar{D},E) \rightarrow \Lambda^{s,t}(\bar{D},E_2),\qquad
(\hat{i^*})^{s,t}:\Lambda^{s,t}(\bar{D},E_2^*) \rightarrow \Lambda^{s,t}(\bar{D},E^*),$$
$$\hat{r}^{s,t}:\Lambda^{s,t}(\bar{D},E_2) \rightarrow \Lambda^{s,t}(\bar{D},E),\qquad
(\hat{r^*})^{s,t}:\Lambda^{s,t}(\bar{D},E^*) \rightarrow \Lambda^{s,t}(\bar{D},E_2^*).$$
Also, these maps act between the corresponding spaces $\Lambda_0^{s,t}$ of forms vanishing on $\partial D$.

\smallskip

First, we prove assertion (1) of the theorem.
To prove that $\mathcal B_0^{n-p,n-q}(\bar{D},E^*)$ is closed in $\mathcal Z_0^{n-p,n-q}(\bar{D},E^*)$, it suffices to construct a continuous linear map
$$
Q_{E}:\mathcal B_0^{n-p,n-q}(\bar{D},E^*) \rightarrow \Lambda_0^{n-p,n-q-1}(\bar{D},E^*), 
$$
such that
\begin{equation}
\label{idQ}
\bar{\partial}Q_{E}=\Id_{\mathcal B_0^{n-p,n-q}(\bar{D},E^*)}. 
\end{equation}
We define
$$Q_{E}:=(\hat{i}^*)^{n-p,n-q-1} \circ Q_{E_2} \circ (\hat{r}^*)^{n-p,n-q},$$
where continuous map $Q_{E_2}:\Lambda^{n-p,n-q}(\bar{D},E_2^*) \rightarrow \Lambda_0^{n-p,n-q-1}(\bar{D},E_2^*)$ for the trivial bundle $E_2$ was constructed in part I of the proof, see \eqref{qe}.
Then, since operator $\bar{\partial}$ commutes with maps $(\hat{i}^*)^{n-p,n-q-1}$ and $(\hat{r}^*)^{n-p,n-q}$, property \eqref{idQ} follows from the analogous one for $Q_{E_2}$ (see above) and in view of the identity  $(\hat{i}^*)^{n-p,n-q-1} \circ (\hat{r}^*)^{n-p,n-q}=\Id_{\Lambda^{n-p,n-q}(\bar{D},E^*)}$. Hence, the  quotient space $H_0^{n-p,n-q}(\bar{D},E^*)$ is Fr\'{e}chet.

\smallskip

Further, identity $r\circ i={\rm Id}_E$ and the definition of (continuous) bilinear form $J_E$ clearly imply for all $\theta \in \mathcal Z^{p,q}(\bar{D},E)$,  $\xi \in \mathcal Z_0^{n-p,n-q}(\bar{D},E^*)$,
\begin{equation}
\label{J_gen}
J_E(\theta,\xi)=J_{E_2}\bigl(\hat{i}^{p,q}(\theta),(\hat{r}^*)^{n-p,n-q}(\xi)\bigr).
\end{equation}
In particular, by Lemma \ref{stokes} for $J_{E_2}$ and the fact that $\bar{\partial}$ commutes with $(\hat{i}^*)^{n-p,n-q-1}$ and $(\hat{r}^*)^{n-p,n-q}$ we have
$J_E(\theta,\xi)=0$ if $\theta \in \mathcal B^{p,q}(\bar{D},E)$ or $\xi \in \mathcal B_0^{n-p,n-q}(\bar{D},E^*)$.
Therefore, $J_E$ descends to a continuous bilinear form 
$$
\mathcal J_E: H^{p,q}(\bar{D},E) \times H_0^{n-p,n-q}(\bar{D},E^*) \rightarrow \mathbb C.
$$
As before, $\mathcal J_E$ determines a continuous linear map $S_E:H_0^{n-p,n-q}(\bar{D},E^*) \rightarrow (H^{p,q}(\bar{D},E))^*$,
\begin{equation}
\label{cohmap_gen}
S_E(h_0):=\mathcal J_E(\cdot, h_0),\qquad h_0\in H_0^{n-p,n-q}(\bar{D},E^*).
\end{equation}
Note that since maps $\hat{i}^{s,t}$, $(\hat{i}^*)^{s,t}$ and $\hat{r}^{s,t}$, $(\hat{r}^*)^{s,t}$ commute with operator $\bar{\partial}$, they descend to maps between the corresponding cohomology groups (denoted  by $\bar{i}^{s,t}$, $(\bar{i}^*)^{s,t}$ and $\bar{r}^{s,t}$, $(\bar{r}^*)^{s,t}$, respectively, and similarly but with the lower index $_0$ in case of maps between $H_0$ cohomology groups).

It follows from \eqref{J_gen}, \eqref{cohmap_gen} and \eqref{cohmap} that
$$
S_E=(\bar{i}^{p,q})^* \circ S_{E_2} \circ (\bar{r}^*)_0^{n-p,n-q}.
$$

Now, consider the second summand in the decomposition $E_2=E\oplus E_1$. Then we have the corresponding embedding and quotient homomorphisms of bundles
$$
i_1:E_1 \rightarrow E_2 \quad\text{and}\quad r_1:E_2 \rightarrow E_1\quad \text{such that}\quad r_1\circ i_1={\rm Id}_{E_1}.
$$
Repeating the above arguments with $(E,i,r)$ replaced by $(E_1,i_1,r_1)$ we arrive to a similar identity for continuous linear maps between the corresponding cohomology groups
$$
S_{E_1}=(\bar{i}_1^{p,q})^* \circ S_{E_2} \circ (\bar{r}_1^*)_0^{n-p,n-q}.
$$
Note that the map
\begin{equation}
\label{cohmap2}
\begin{array}{l}
\bigl((\bar{i}^{p,q})^*,(\bar{i}_1^{p,q})^* \bigr) \circ S_{E_2} \circ \Sigma\bigl((\bar{r}^*)_0^{n-p,n-q},(\bar{r}_1^*)_0^{n-p,n-q} \bigr): \\
\\
H_0^{n-p,n-q}(\bar{D},E^*) \oplus H_0^{n-p,n-q}(\bar{D},E_1^*) \rightarrow (H^{p,q}(\bar{D},E))^* \oplus (H^{p,q}(\bar{D},E_1))^*,
\end{array}
\end{equation}
where $\Sigma(u,v)=u+v$ ($u,v \in H_0^{n-p,n-q}(\bar{D},E_2^*)$), is an isomorphism.\\ Indeed, by the result of part I map $S_{E_2}:H_0^{n-p,n-q}(\bar{D},E_2^*) \rightarrow (H^{p,q}(\bar{D},E_2))^*$ is an isomorphism. Also, decomposition
$E \oplus E_1=E_2$ implies that maps $\bigl((\bar{i}^{p,q})^*,(\bar{i}_1^{p,q})^* \bigr)$ and $\Sigma\bigl((\bar{r}^*)_0^{n-p,n-q},(\bar{r}_1^*)_0^{n-p,n-q} \bigr)$ are isomorphisms between the corresponding spaces.

Next, by the definition of bilinear form $J_{E_2}$ we have 
\[
(\bar{i}_1^{p,q})^* \circ S_{E_2} \circ (\bar{r}^*)_0^{n-p,n-q}=
(\bar{i}^{p,q})^* \circ S_{E_2} \circ (\bar{r}_1^*)_0^{n-p,n-q}=0.
\]
Therefore, isomorphism \eqref{cohmap2} coincides with $S_{E} \oplus S_{E_1}$. This implies, in particular, that $S_E$ is an isomorphism and completes the proof of part (1) of the theorem.

 \medskip

Let us prove (2)(b). Let $M:=M_{E_2}$ be the map of part (2)\,(b) of the theorem for the trivial bundle $E_2$. We set
$$N_{E_2}:=M_{E_2}^{-1}: H_0^{n-p,n-q}(\bar{D},E_2^*) \rightarrow (B_2^*)^m$$ 
and define a continuous linear map
$$
N_{E}: H_0^{n-p,n-q}(\bar{D},E^*) \rightarrow \mathcal O(X,E^*)^m
$$
by the formula
$$
N_{E}:=\hat{i}^* \circ N_{E_2} \circ (\bar{r}^*)_0^{n-p,n-q},
$$
where $\hat{i}^*:=\oplus^m \bigl((\hat{i}^*)^{0,0}|_{B^*_2} \bigr)$ (here $B^*_2$ is identified with the space of constant sections in $\mathcal O(X,E^*_2)$). Since $N_{E_2}$ is continuous, map $N_{E}$ is continuous as well.

Let us show that $N_{E}$ is injective. We argue as above.
Namely, map 
\begin{equation}
\label{mapinj}
(\hat{i}^*, \hat{i}_1^*) \circ N_{E_2} \circ \Sigma\bigl((\bar{r}^*)_0^{n-p,n-q},(\bar{r}^*)_0^{n-p,n-q}\bigr):
\end{equation} 
$$
H_0^{n-p,n-q}(\bar{D},E^*) \oplus H_0^{n-p,n-q}(\bar{D},E_1^*) \rightarrow \mathcal O(X,E^*)^m \oplus \mathcal O(X,E_1^*)^m
$$
is injective. Indeed,
$N_{E_2}$ is an isomorphism by the corresponding result of part I, and the injectivity of maps
$(\hat{i}^*, \hat{i}_1^*)$, $\Sigma\bigl((\bar{r}^*)_0^{n-p,n-q},(\bar{r}^*)_0^{n-p,n-q}\bigr)$ follow from the decomposition $E \oplus E_1=E_2$.
Since 
\[
\hat{i}_1^* \circ N_{E_2} \circ (\bar{r}^*)_0^{n-p,n-q}=\hat{i}^* \circ N_{E_2} \circ (\bar{r}_1^*)_0^{n-p,n-q}=0
\]
(because $r_1\circ i=r\circ i_1=0$),
injective map \eqref{mapinj} coincides with $N_E \oplus N_{E_1}$, and so map $N_{E}$ must be injective as well.

Now, we define $$\mathcal B:=N_{E} \bigl( H_0^{n-p,n-q}(\bar{D},E^*)\bigr) \subset \mathcal O(X,E^*)^m.$$ 
(Space $\mathcal B \subset \mathcal O(X,E^*)^m$ is endowed with the Fr\'{e}chet topology of uniform convergence on compact subsets of $X$.)
Let us show that $\mathcal B$ is a closed subspace of $\mathcal O(X,E^*)^m$. To this end, we define a continuous linear map
$$
M=M_{E}:=(\bar{i}^*)_0^{n-p,n-q} \circ \widetilde M_{E_2} \circ \hat{r}^*:\mathcal B \rightarrow H_0^{n-p,n-q}(\bar{D},E^*), \quad \hat{r}^*=\oplus^m (\hat{r}^*)^{0,0},
$$
where 
\[
\widetilde M_{E_2}(h_1,\dots,h_m):=\left[\sum_{i=1}^m h_i|_{\bar D} \cdot \gamma_i\right], \quad (h_1,\dots,h_m) \in \mathcal O(X,E_2^*)^m.
\]
By definition, $M_{E_2}=\widetilde M_{E_2}|_{(B_2^*)^m}$. Also, one can easily check that 
\[
(\bar{i}^*)_0^{n-p,n-q} \circ \widetilde M_{E_2} \circ \hat{r}^*\circ\hat{i}^*|_{(B_2^*)^m}=(\bar{i}^*)_0^{n-p,n-q} \circ M_{E_2}.
\]
From here, using that $M_{E_2} \circ N_{E_2}=\Id_{H_0^{n-p,n-q}(\bar{D},E_2^*)}$, we obtain 
$M_{E} \circ N_{E}=\Id_{H_0^{n-p,n-q}(\bar{D},E^*)}$. Since $M_E$ is continuous, the latter identity implies that space $\mathcal B$ is complete and hence  is closed in $\mathcal O(X,E^*)^m$.

The fact that $\mathcal B$ is isomorphic to the dual of $\mathcal A$ is now immediate, since by what we have proved above $\mathcal B \cong (H^{p,q}(\bar{D},E))^*$, while by Theorem \ref{thm2}(1), $\mathcal A \cong H^{p,q}(\bar{D},E)$. The proof of assertion (2)(b) is complete.

The proof of assertion (2)(a) is analogous to the proof of part (1)(a) of Theorem \ref{thm2}.

The proof of the theorem is complete.


\begin{thebibliography}{99}

\bibitem[Br]{Br}
A.~Brudnyi,  Holomorphic functions of slow growth on coverings of pseudoconvex domains in Stein manifolds. {\em Compositio Math.} 142 (2006), 1018--1038.

\bibitem[BrK]{BrK}
A.~Brudnyi and D.~Kinzebulatov,
\newblock Towards Oka-Cartan theory for algebras of holomorphic functions on coverings of Stein manifolds I, II.
\newblock {\em Revista Mat.~Iberoamericana} (to appear), 2013.

\bibitem[Bun]{Bun}
L.~{Bungart},
\newblock On analytic fibre bundles {I.} {H}olomorphic fibre bundles with
  infinite dimensional fibres.
\newblock {\em Topology} 7 (1968), 55--68.


\bibitem[FK]{FK}
G.~B.~Folland and J.~J.~Kohn,
\newblock The Neumann Problem For The Cauchy-Riemann Complex.
\newblock {\em Princeton}, 1972.

\bibitem[GH]{GH} 
P.~Griffiths and J.~Harris,
\newblock Principles of Algebraic Geometry.
\newblock{\em Wiley}, 1994.

\bibitem[He]{He}
M.~Hestens, 
\newblock Extension of the range of a differentiable function.
\newblock {\em Duke Math. J.}
8 (1941), 183--192.

\bibitem[K]{K}
J.~J.~Kohn,
\newblock Harmonic integrals on strongly pseudoconvex manifolds, I.
\newblock {\em Ann. of Math.} 78 (1963), 112--148, ibid 79 (1964), 450--472.

\bibitem[KN]{KN}
J.~J.~Kohn and L.~Nirenberg,
\newblock Non-coercive boundary value problems.
\newblock {\em Comm. Pure Appl. Math.} 18 (1965), 443--492.


\bibitem[M]{M}
B. S.~Mitjagin, \newblock The homotopy structure of the linear group of a Banach space.
\newblock {\em Russian Math. Surveys} 25 (1970), 59--103.

\bibitem[S]{Shab}
B. V.~Shabat,
\newblock Introduction to Complex Analysis: Functions of Several Variables.
\newblock {\em AMS}, 1992.


\bibitem[R]{R}
W.~Rudin, \newblock Principles of Mathematical Analysis.
\newblock{\em McGraw-Hill}, 1976.

\bibitem[ZKKP]{Obz}
M.~Zaidenberg, S.~G.~Krein, P.~Kuchment and A.~Pankov,
\newblock Banach bundles and linear operators.
\newblock {\em Russian Math. Surveys} 30 (1975), 115--175.


\end{thebibliography}
\end{document}